\documentclass{llncs}
\usepackage[T1]{fontenc}
\usepackage[utf8]{inputenc}
\usepackage{amsmath,amssymb} 
\usepackage{graphicx,amsfonts,xcolor,tikz,hyperref,float}
\usepackage{mathtools}
\usepackage{booktabs}
\usepackage{enumitem}
\usepackage{orcidlink}
\usepackage{xcolor}

\hypersetup{pdftitle={A bounded-defect construction for the Mondrian art problem},pdfsubject={Pell families of defect 280 and 560}}

\newcommand{\Q}{\mathbb Q}
\newcommand{\N}{\mathbb N}
\newcommand{\defect}{\operatorname{def}}
\newcommand{\blfootnote}[1]{%
  \begingroup
  \renewcommand{\thefootnote}{}%
  \footnotetext[0]{#1}%
  \endgroup
}

\title{A Bounded-Defect Construction for the Mondrian Art Problem}
\author{Cooper O'Kuhn \and Freeman Slaughter\orcidlink{0009-0000-5787-6183}}
\institute{ University of South Florida, Tampa FL, USA \\
\email{\{cokuhn, fslaughter\}@usf.edu} 
}
\date{October 2026}

\begin{document}

\maketitle
\vspace{-1.5em}

\begin{abstract}
We construct an infinite family of solutions to the Mondrian art problem with absolute defect bounded above by $280$, of side lengths $n_0 = 315$, $n_1 = 106785$, and $n_{j+2} = 340n_{j+1} - n_j$, where $n_j$ is congruent to $315 \pmod{630}$. 
This absolute bound shows that the minimum defect does not necessarily increase with $n$, providing a counterexample to a conjecture which has been open since the problem's original proposal.
Our design is derived from a well-known seven-rectangle dissection over the real quadratic field $\Q(\sqrt{19})$, with non-integer side lengths;
we tailor the side lengths to be integers via an infinite Pell family.
Using the same method, we devise a second family with defect $560$ that contains infinitely many side lengths which are divisible by any given positive integer, implying that for any $k$ there are infinitely many side lengths $n \equiv 0 \pmod{k}$ with defect bounded above by $560$.
This constant bound offers a dramatic improvement on the previous quadratic bounds.

\blfootnote{Date: \today.}
\blfootnote{This research was partially supported by NSF Grant DMS-2342254.}
\end{abstract}

\section{Introduction}
The Mondrian art problem is named after the Dutch painter Piet Mondrian who is best known for his De Stijl artworks primarily composed of colored rectangles. The problem asks how closely the areas of rectangles partitioning a square canvas can agree when repeated tiles are forbidden \cite{OF}. 
More precisely, for an integer $n \geq 3$, consider a partition $\mathcal P$ of the square canvas $[0,n] \times [0,n]$ into at least two axis-aligned rectangles with positive integer side lengths. 
The rectangles must be pairwise noncongruent, including after rotation. 
Thus, for example, a $2 \times 3$ rectangle and a $3 \times 2$ rectangle cannot appear together on the same canvas, due to this noncongruence restriction.
Denoting the area of a rectangle $R$ with $|R|$, we can then define the defect
\[
    \defect(\mathcal P) = \max_{R\in\mathcal P} |R| - \min_{R\in\mathcal P} |R|.
\]
The minimum defect, $M(n) = \min\limits_{\mathcal P} \defect(\mathcal P)$, is taken over all admissible partitions of the canvas.

A partition is called \emph{perfect} if achieves a defect of zero: its rectangles all have equal area, despite being pairwise noncongruent.
The central zero-defect question asked in \cite{OF} seeks to determine for which $n \geq 3$ the minimal defect satisfies $M(n) = 0$, if any. 
Equal areas alone are insufficient, since the corresponding integer-sided rectangles must also fit together with no gaps or overlaps.

The consensus is generally that $M(n)$ somehow increases with $n$, suggesting that no such integer $n$ satisfies $M(n) = 0$, though its exact asymptotic rate has been the subject of much consideration. 
We present a construction that dispels this notion: in fact, $M(n)$ can be bounded above along an infinite sequence by an absolute constant, independent of $n$.



\begin{theorem}\label{thm:main}
Define the recurrence relation
\[
    n_0 = 315, \quad n_1 = 106785, \quad n_{j+2} = 340n_{j+1} - n_j.
\]
For every $j \geq 0$, the $n_j \times n_j$ square has a Mondrian partition into seven pairwise noncongruent integer rectangles with defect $280$. 
Moreover, for all $j$,
\[
    n_j \equiv 315 \pmod{630}.
\]
Consequently $\liminf\limits_{n\to\infty} M(n) \leq 280$.
\end{theorem}

Our starting point is the real equal-area seven-rectangle dissection from Dalf\'o, Fiol, L\'opez, and Mart\'inez-P\'erez \cite{DFLM}. 
The coordinates of this construction lie in the real quadratic field $\Q(\sqrt{19})$. 
Lebowitz-Lockard and Tenenbaum \cite{LLT} reproduce this dissection, but attribute it to the pseudonymous Blanche Descartes \cite{Descartes}. 
We continue this line of work by utilizing this construction, specialized to solutions of a Pell equation.

Again, it is still an open problem whether $M(n) = 0$ for some $n$, though it is conjectured that no such $n$ exists. 
The first author and Fellman \cite{OF} gave one of the initial lower bounds for the number of $n \leq x$ satisfying $M(n) > 0$ of order $\tfrac{x}{\log\log(x)}$, improved from an earlier preprint obtaining $\tfrac{x\log\log(x)}{\log(x)}$, also by the first author \cite{O}. 
They additionally observed that a single perfect partition would imply an infinite family, as scaling all sides by the same multiple leads to a larger solution with the same zero defect. 
Thus, to show that no such integers exist, it suffices to show that the set of such integers $n \leq x$ is at most density 0. 
More recently, Garc\'ia-Col\'in et al. \cite{GCLMR} exclude perfect partitions for all $n \leq 1000$. 
An August 2026 manuscript by Lebowitz-Lockard et al. \cite{LLT} shows that a positive proportion of integers have nonzero defect, and in September 2026, a preprint by Garrison et al. \cite{GSS} was the first to give a non-trivial upper bound on the defect itself: that $M(n) \leq 10n^{5/6}$ for all sufficiently large $n$. 

The prevailing sentiment was that $M(n) \neq 0$ for all $n$ \textit{simply as a consequence} of the observation that $M(n)$ grows with $n$, a trend reflected in known examples with $n \leq 900$ from \cite{Pegg}. Theorem~\ref{thm:main} disproves this, and thus if this zero-defect conjecture is true, it requires a more delicate line of reasoning.

\section{The Construction}
Take $p,q \in \N$ such that
\begin{equation}\label{eq:conditions}
    q < p < 7q, \quad p \equiv q \pmod 2.
\end{equation}
Set $n = 105q$ and define the parameters
\begin{equation}\label{eq:cuts}
    \begin{aligned}
        a &= 7(7q-p), & b &=42q, & c &= 7(7q+p), \\
        d &= 5(8q-p), & e &=5(13q-p), & f &= \frac{15}{2}(7q+p).
    \end{aligned}
\end{equation}
Note that all of these coordinates are integers, since $p$ and $q$ have the same parity. 
The inequalities in \eqref{eq:conditions} imply an ordering of parameters,
\[
    0 < a < b < c < n, \quad 0 < d < e < f < n.
\]
This can be readily seen by taking pairwise differences; for example, $b - a = 7(p-q) > 0$, $n - f = \frac{15}{2} (7q-p) > 0$, and $f - e = \frac{25}{2} (p-q) > 0$.

Define seven rectangles by their intervals as follows:
\begin{equation}\label{eq:rectangles}
    \begin{aligned}
        R_1 &= [a,n] \times [0,d], & R_2 &= [0,a] \times [0,f], \\
        R_3 &= [0,c] \times [f,n], & R_4 &= [c,n] \times [e,n], \\
        R_5 &= [b,n] \times [d,e], & R_6 &= [a,b] \times [d,f], \\
        R_7 &= [b,c] \times [e,f].
    \end{aligned}
\end{equation}
These rectangles have disjoint interiors and cover the square, hence they form a partition of the canvas. 
Indeed, in the successive horizontal bands with endpoints $0,d,e,f,n$, the left-to-right lists of rectangles going from bottom to top are
\[
    (R_2,R_1), \quad (R_2,R_6,R_5), \quad (R_2,R_6,R_7,R_4), \quad (R_3,R_4).
\]
See Figure~\ref{fig:315 example} for our generic Mondrian dissection. 

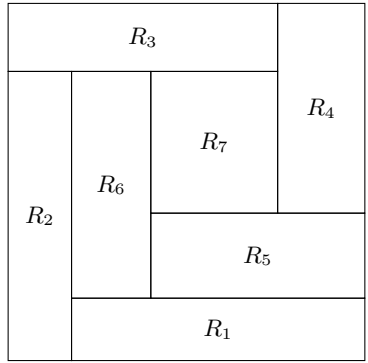
\begin{figure}
    \centering
    \begin{center}
    \begin{tikzpicture}[x=0.015cm,y=0.015cm,every node/.style={font=\small}]
    \draw (56,0) rectangle (315,55);
    \draw (0,0) rectangle (56,255);
    \draw (0,255) rectangle (238,315);
    \draw (238,130) rectangle (315,315);
    \draw (126,55) rectangle (315,130);
    \draw (56,55) rectangle (126,255);
    \draw (126,130) rectangle (238,255);
    \node at (185.5,27.5) {$R_1$};
    \node at (28,127.5) {$R_2$};
    \node at (119,285) {$R_3$};
    \node at (276.5,222.5) {$R_4$};
    \node at (220.5,92.5) {$R_5$};
    \node at (91,155) {$R_6$};
    \node at (182,192.5) {$R_7$};
    \node[below] at (157.5,0) {};
    \end{tikzpicture}
    \end{center}
    \caption{The Mondrian canvas for our generic construction. Note the superficial similarities to the Blanche dissection \cite{Descartes}.}
    \label{fig:315 example}
\end{figure}

Their widths and heights, which are integers by the argument above, are as follows:
\begin{center}
    \renewcommand{\arraystretch}{1.3}
    \begin{tabular}{c@{\quad}c@{\quad}c}
    \toprule
    Rectangle & Width & Height\\
    \midrule
    $R_1$ & $7(8q+p)$ & $5(8q-p)$ \\
    $R_2$ & $7(7q-p)$ & $\frac{15}{2}(7q+p)$ \\
    $R_3$ & $7(7q+p)$ & $\frac{15}{2}(7q-p)$ \\
    $R_4$ & $7(8q-p)$ & $5(8q+p)$ \\
    $R_5$ & $63q$ & $25q$ \\
    $R_6$ & $7(p-q)$ & $\frac{25}{2}(p+q)$ \\
    $R_7$ & $7(p+q)$ & $\frac{25}{2}(p-q)$ \\
    \bottomrule
    \end{tabular}
\end{center}

For $D \in \N$ nonsquare, Pell's equation \cite{Lenstra,Stillwell} is defined as 
\[
    u^2 - Dv^2 = 1,
\]
where integer solutions $u,v$ are sought. 
Every positive solution can be used to generate infinitely many, by taking powers of $u + v \sqrt D$ and collecting terms. 
More generally, multiplying $p + q \sqrt D$ by $u + v \sqrt D$ preserves the value of $p^2 - Dq^2$.
This permits our parameters to grow while still satisfying the same fixed quadratic relation, which allows our construction to maintain a constant area defect over all parameter regimes.

\begin{proposition}\label{prop:area}
    For the integer dissection above, consider the Pell equation
    \[
        E = p^2 - 19q^2.
    \]
    Then all seven rectangles are pairwise noncongruent, and for the partition constructed as $\mathcal{P} = \{R_1, \dots, R_7\}$ in Figure \ref{fig:315 example},
    \begin{equation}\label{eq:defect}
        \defect(\mathcal{P}) = 140|E|.
    \end{equation}
\end{proposition}
\begin{proof}
    Define $B = 1575q^2 = \tfrac{n^2}{7}$.
    Direct multiplication gives
    \begin{equation}\label{eq:areas}
        \begin{aligned}
            |R_1| = |R_4| &= 35(64q^2-p^2) = B - 35E, \\
            |R_2| = |R_3| &= \frac{105}{2} (49q^2 - p^2) = B - \frac{105}{2}E, \\
            |R_5| &= 1575q^2 = B, \\
            |R_6| = |R_7| &= \frac{175}{2} (p^2-q^2) = B + \frac{175}{2}E.
        \end{aligned}
    \end{equation}
    Note that $E \neq 0$, since $\sqrt{19}$ is irrational, and $q>0$.
    The four coefficients of $E$ have a maximum difference of
    \[
        \frac{175}{2} - \left(-\frac{105}{2}\right) = 140,
    \]
    which proves \eqref{eq:defect}. \cite{Stillwell}
    
    Within each of the three equal-area pairs, the side lengths may be parameterized as \((\alpha u, \beta v)\) and \((\alpha v, \beta u)\), where \(u \neq v\) and \(\alpha \neq \beta\). 
    Their perimeters differ by \(2(\alpha - \beta)(u - v) \neq 0\), so the rectangles in this construction cannot be congruent. 
    For $(R_1, R_4)$ take $(\alpha, \beta) = (7, 5)$ and $(u, v) = (8q + p, 8q - p)$, then for $(R_2, R_3)$ take $(\alpha, \beta) = (7, 15/2)$ and $(u, v) = (7q - p, 7q + p)$, 
    and finally for $(R_6, R_7)$ take $(\alpha, \beta) = (7, 25/2)$ and $(u, v) = (p - q, p + q)$. 
    Thus, none of the equal area pairs are congruent, and clearly none can be congruent to $R_5$. 
\end{proof}

\begin{example}
    For the base case $(p, q) = (13, 3)$, the rectangle lengths are
    \[
        (a, b, c) = (56, 126, 238), \quad (d, e, f) = (55, 130, 255),
    \]
    which corresponds to canvas length $n = 315$.
    The rectangle dimensions, respectively, are then
    \[
        259 \times 55, \quad 56 \times 255, \quad 238 \times 60, \quad 
        77 \times 185, \quad 189 \times 75, \quad 70 \times 200, \quad 112 \times 125.
    \]
    The largest area is $238 \times 60 = 14280$, while the smallest is $70 \times 200 = 14000$, so the defect of this partition is their difference: $280$.
    This explicit example shows that $M(315) \leq 280$, but it may not be the optimal construction for $n = 315$.
\end{example}

\section{Proof of Main Theorem}
\begin{proof}[Proof of Theorem~\ref{thm:main}]
    Define
    \begin{equation}\label{eq:pell}
        p_j + q_j \sqrt{19} = (13 + 3\sqrt{19}) (170 + 39\sqrt{19})^j.
    \end{equation}
    or equivalently, $(p_0, q_0) = (13, 3)$ and 
    \begin{equation}\label{eq:pell recursive}
        p_{j+1} + q_{j+1} \sqrt{19} = (170 + 39\sqrt{19}) (p_{j} + q_{j}\sqrt{19}).
    \end{equation}
    
    Expanding and equating coefficients in \eqref{eq:pell recursive} uncovers the recurrence relation
    
    \begin{equation}\label{eq:recurrence}
        \begin{pmatrix}
            p_{j+1} \\
            q_{j+1}  
        \end{pmatrix}
        =
        \begin{pmatrix}
            170 & 741 \\
            39 & 170
        \end{pmatrix}
        \begin{pmatrix}
            p_{j} \\
            q_{j}  
        \end{pmatrix}
    \end{equation}
    
    Now, leveraging the multiplicativity of the norm on the real quadratic field $\Q[\sqrt{19}]$ together with the simple identities
    \[
        13^2 - 19 \cdot 3^2 = -2 \ \text{ and } \ 170^2 - 19 \cdot 39^2 = 1,
    \]
    we obtain, from \eqref{eq:pell},
    \[
        p_j^2 - 19q_j^2 = -2.
    \]
    
    Starting from the initial $(13, 3)$, both parameters remain positive and odd. 
    Additionally, since $q_{j+1} \equiv 2q_j \pmod 3$, it's straightforward to see that $3 \mid q_j$ for every $j$, hence $q_j \equiv 3 \pmod 6$.
    
    The norm equation implies $q_j < p_j < 7q_j$, so our construction applies for each $j$, thus Proposition~\ref{prop:area} gives that each have defect $280$. 
    More explicitly, the four distinct areas are
    \[
        B - 175, \quad B, \quad B + 70, \quad B + 105.
    \] 
    Setting $n_j = 105q_j$, we have that $n_j \equiv 315 \pmod{630}$. 
    Since the recurrence matrix in \eqref{eq:recurrence} has trace $340$ and determinant $1$, each pair satisfies the second-order recurrence provided.
    Since the recurrence given in \eqref{eq:pell recursive} is strictly increasing, the sequence of $q_j$, and hence $n_j$, grow without bound.
\end{proof}

The first few initial values for these parameters are highlighted in Table \ref{tab:Initial}.


\begin{table}
    \centering
    \setlength{\tabcolsep}{2mm}
    \begin{tabular}{c||cccc}
          & $0$ & $1$ & $2$ & $3$  \\
         \hline \hline
         $(p_j, q_j)$ & $(13,3)$ & $(4433,1017)$ & $(1507207,345777)$ & $(512445947,117563163)$  \\
         $n_j$ & $315$ & $106785$ & $36306585$ & $12344132115$    \\
    \end{tabular}
    \caption{Initial values of our construction, with $j$ on top.}
    \label{tab:Initial}
\end{table}

By standard results on Pell equations, this sequence grows exponentially with power $\alpha = 170 + 39\sqrt{19}$, hence its counting function is
\[
    \#\{j \mid n_j \leq x\} = \frac{\log(x)}{\log(\alpha)} + O(1).
\]
Our construction demonstrates the bound is attained for infinitely many side lengths in this congruence class, but not that it holds for every sufficiently large side length within the class.

One might ask what congruence classes these side lengths fall into; this next results shows that this family sits in a particularly simple class for all congruences $k$.

\begin{corollary}\label{cor:residue}
    For every fixed positive integer $k$, infinitely many members of this defect-$280$ family satisfy $n_j \equiv 315 \pmod k$.
\end{corollary}
\begin{proof}
    The recurrence matrix in \eqref{eq:recurrence}, call it $A$, has determinant 1, so is invertible modulo $k$ for every $k$. 
    Taking repeated powers of $A$ forms a finite matrix group, so there exists a positive power $N$ such that $A^N$ is the identity modulo $k$. 
    Along the subsequence iterated by $N$, we have $(p_j,q_j) \equiv (13,3) \pmod k$, and therefore $n_j \equiv 105 \cdot 3 = 315 \pmod k$.
\end{proof}

Suppose one starts with a Mondrian partition of the $n \times n$ square. 
Multiplying every side length by $k \in \N$ provides the trivial bound $M(kn) \leq k^2 M(n)$, which immediately shows why a single perfect partition provides infinitely many. 
As a consequence, the defect-$280$ family provided implies that for all fixed $k$, 
\[
    \liminf_{r \rightarrow \infty} M(kr) \leq 280 k^2.
\]
It is an open line of research to determine if it is possible to remove the quadratic dependence of $k$ from this equation, by replacing it with a slower-growing function. The best in the literature seems to be the spiral family constructed in \cite{DFL}, for instance, with $n = \tfrac{k!}{2}$ and defect $\tfrac{[(k-1)!!]^2}{2}$, a slight improvement over this quadratic dependence. This next corollary shows that in fact one can completely remove the contribution of $k$ at the cost of doubling the constant defect-$280$ bound, namely that for all $k \in \N$,
\[
    \liminf_{r \rightarrow \infty} M(kr) \leq 560.
\]

\begin{corollary}\label{cor:multiples}
    For every positive integer $k$, there exist infinitely many $n \equiv 0 \pmod k$ such that $M(n) \leq 560$.
    In other words, defining $n = kr$, we can write this as
    \[
        \sup_{k \geq 1} \liminf_{r \rightarrow \infty} M(kr) \leq 560.
    \]
\end{corollary}
\begin{proof}
    For $j \in N$, define integers $u_j, v_j$ by $u_j + v_j\sqrt{19} = (170 + 39 \sqrt{19})^j$.
    It does not necessarily hold anymore that $u_j$ and $v_j$ have the same parity, so we take $p = 2u_j$ and $q = 2v_j$ to ensure that our construction is valid, though this increases the norm to $E = 4$.
    These parameters still satisfy \eqref{eq:conditions}, so their square sides $n = 210v_j$ have defect $560$. 
    The recurrence matrix $A$ has a positive power $N$ such that $A^N$ collapses to the identity modulo $k$; taking $j$ to be multiples of $N$ yields $(u_j, v_j) \equiv (1, 0) \pmod k$. 
    These square sides are divisible by $k$, and grow without bound, hence the claimed result.
\end{proof}

An immediate question is the optimality of this construction.
We now prove that this particular parameterization is optimal within its family.

\begin{proposition}
    Among all integer dissections satisfying \eqref{eq:conditions}, the minimal defect attainable for this type of partition is $280$. 
    In particular, it can never achieve defect zero.
\end{proposition}
\begin{proof}
    If $p, q$ are odd, then
    \[
        p^2 - 19q^2 \equiv 2 \pmod{4},
    \]
    so $|E| \geq 2$. 
    If they are both even, then $4 \mid E$, so $E \neq 0$ implies $|E| \geq 4$. 
    Proposition~\ref{prop:area} therefore provides a defect that is at least $280$. 
    Equality is attained by our construction, the norm-$(-2)$ family.
\end{proof}

We wish to reiterate that this does not prove directly a lower bound for $M(n)$, and is not a lower bound over all seven-rectangle tilings. 
However, it does rules out a defect below $280$ -- and in particular, a zero defect -- from picking successive rational approximants inside the family.

\section{Concluding Remarks and Future Directions}
In this paper, we presented an infinite construction that achieves a bounded absolute defect of $280$ for each instance, and showed that this is the smallest defect realizable by partitions of this type.
We additionally showed that for each positive integer $k$, there are infinitely many side lengths $n \equiv 0 \pmod{k}$ with a defect bounded by $560$.
An immediate consequence of this construction is that $M(n)$ does not grow without bound;
more generally, there cannot even exist a divergent function $f : \N \rightarrow \mathbb{R}$ (i.e.: linear, logarithmic, etc.) and positive integer threshold $N$ such that for all $n \geq N$, the defect satisfies $M(n) \geq f(n)$.


We foresee a quick flurry of progressive improvements. 
We intend to follow up on this work with distinct constructions that enjoy the same defect, or -- if possible -- perform better.
There are several more promising angles in the $M(n) = 0$ direction we intend on pursuing, though such a construction would have to lie outside the method of finding increasingly accurate approximations to a configuration with irrational side length.

We conjecture that our construction possesses the lowest defect that one can hope to achieve, among dissections of the same type.
\begin{conjecture}
    For partitions with seven pairwise noncongruent integer-sided rectangles, we conjecture that the above defect bound is optimal, in the sense that 
    \[\liminf\limits_{n\to\infty} M(n) = 280.\]
\end{conjecture}
Such a proof requires more advanced arguments however, so we relegate this to a later paper.

An interesting line of future work is to consider a generalized version of the Mondrian art problem wherein the canvas need not be square, say a $[0,m] \times [0,n]$ rectangle, then ask what bounds are possible for the defect $M(m,n)$.  
In \cite{GCLMR} for instance, it is shown computationally that no perfect rectangular Mondrian partitions exist for $m,n \leq 400$.
We strongly believe that the above seven-rectangle construction can lead to a greatly decreased generalized defect for the rectangular Mondrian art problem, but leave this for future work.
\begin{conjecture}
    We conjecture that the rectangular analogue of the Mondrian art problem correspondingly attains a defect bound of the form 
    \[\liminf\limits_{m,n\to\infty} M(m,n) \leq c,\]
    where $c \in \N$ does not depend on $m$ or $n$, but satisfies $c < 280$.
    Namely, that the rectangular Mondrian problem achieves a lower absolute bound on the defect than the square version.
\end{conjecture}

\bibliographystyle{splncs04}
\bibliography{archive.bib} 
\end{document}